\documentclass[reqno,12pt]{amsart}
\usepackage[T1]{fontenc}
\usepackage[utf8]{inputenc}
\usepackage[english]{babel}
\usepackage{amssymb,amsmath,amsthm,amscd,latexsym,amsfonts,xcolor,enumerate,hyperref,comment,longtable,cleveref}

\usepackage{times}
\usepackage{cite}
\usepackage{pdflscape}
\usepackage[mathcal]{euscript}
\usepackage{tikz}
\usepackage{hyperref}
\usepackage{cancel}
\usepackage{stmaryrd}
\usepackage{dsfont}

\numberwithin{equation}{section}

\usetikzlibrary{arrows}

\DeclareMathOperator{\gr}{gr}

\usepackage[a4paper,top=3cm, bottom=3cm, left=3cm, right=3cm]{geometry}

\makeatletter
\@namedef{subjclassname@2020}{%
  \textup{2020} Mathematics Subject Classification}
\makeatother

\newtheorem{pr}[subsection]{Proposition}

\newtheorem{teo}[subsection]{Theorem}

\theoremstyle{definition}
\newtheorem{de}[subsection]{Definition}

\theoremstyle{remark}

\newtheorem{remark}[subsection]{Remark}

\usepackage{stmaryrd}

\begin{document}

\title[Naturally graded nilpotent associative algebras of nilindex $n-3$]
{Naturally graded nilpotent associative algebras of nilindex $n-3$}

	\author{I.A.Karimjanov}
\address{Department of Mathematics, Andijan State University, 129, Universitet Street, Andijan,
170100, Uzbekistan,
\newline and\newline
V.I.Romanovskiy Institute of Mathematics, Uzbekistan Academy of Sciences, Univesity Street, 9, Olmazor district, Tashkent, 100174, Uzbekistan,
\newline
		and\newline
		Saint-Petersburg State University, Saint-Petersburg, Russia.}

\email{iqboli@gmail.com}

\thanks{This work was partially supported by RSF 22-71-10001.}

\begin{abstract} In the present paper, we give the classification of a subclass of n-dimensional naturally graded associative algebras with nilindex $n-3$. The subclass has the characteristic sequence $C(\mathcal{A})=(n-3,2,1)$. The result completes the classification of naturally graded associative algebras with nilindex $n-3$ for $n>6$.
\end{abstract}

\subjclass[2020]{16S50, 16W50}
\keywords{associative algebras, nilpotent, naturally graded, characteristic sequence, left multiplication operator}

\maketitle

\section{Introduction}

In general, the classification of finite-dimensional nilpotent associative algebras up to isomorphism is very complicated and still an open problem.  Due to the difficulties of obtaining a complete classification there are two options: to classify all nilpotent algebras of small dimension or arbitrary dimensional algebras with given some properties. A lot of progress has been made in the classifications of low-dimensional nilpotent associative algebras. Much research has been devoted to classification of low-dimensional associative algebras \cite{Haz, kar, Eick, Kr,  Eick2, Graaf, Maz1, Maz2, kay, Poo}.

 Finite-dimensional nilpotent associative algebras with different properties have been investigated by many authors \cite{sil,sm,el}. We remark that $n-$dimensional nilpotent associative algebra has a nilpotency index less or equal $n$. A nilpotent associative algebras which has a maximum nilindex are one-generated and called null-filiform algebras. There is unique such type associative algebra up to isomorphism for fixed dimension.

The next step is classification of nilpotent associative algebras with nilindex $n-1$ and $n-2$. Such type algebras called filiform\cite[see Proposition 3.4]{kar1} and quasi-filiform algebras respectively. The complete classification of filiform and quaisi-filiform associative algebras up to isomorphism have been done in\cite{kar1,kar2}.


The aim of the present paper is to classify naturally graded $n-$ dimensional associative algebras of nilindex $n-3$. It is necessary to consider the next characteristic sequences: $(n-3,3)$, $(n-3,1,1,1)$ and $(n-3,2,1)$ for the classification of naturally graded associative algebras of nilindex $n-3$. The classification of naturally graded associative algebras with characteristic sequences $(n-3,3)$ and $(n-3,1,1,1)$ has been done in \cite{kar3,kar4}. Thus, it is sufficient to consider naturally graded associative algebras with characteristic sequence $(n-3,2,1)$ to complete the study.

In this paper we present classification up to isomorphism naturally graded $n-$ dimensional associative algebras with characteristic sequence of $(n-3,2,1)$. The result completes the classification of naturally graded associative algebras with nilindex $n-3$ for $n>6$. Similar results for Leibniz and Zinbiel algebras were obtained in the works \cite{red, ad}. Actually, the classification of naturally graded nilpotent algebras with nilindex less or equal $n-3$ have been completed for the cases Lie, Leibniz and Zinbiel algebras. Throughout the paper algebras are finite-dimensional over the field of the complex numbers.

\section{Preliminaries}\label{S:prel}

In this section we give necessary definitions and preliminary results.

For an associative algebra $\mathcal{A}$, over field of complex number, we define series:
\[
\mathcal{A}^1=\mathcal{A}, \qquad \ \mathcal{A}^{i+1}=\mathcal{A}^i\mathcal{A}, \qquad i\geq 1.
\]

If $\mathcal{A}^{s}=0$ for some $s \in \mathbb{N}$ then an algebra $\mathcal{A}$ is called \emph{nilpotent}. The integer $k$ satisfying $\mathcal{A}^{k}\neq0$ and $\mathcal{A}^{k+1}=0$ is called the \emph{index of nilpotency} or \emph{nilindex} of $\mathcal{A}$.

 Given a finite-dimensional nilpotent associative algebra $\mathcal{A}$, put
$\mathcal{A}_i=\mathcal{A}^i/\mathcal{A}^{i+1}, \ 1 \leq i\leq k$, and $\gr \mathcal{A} = \mathcal{A}_1 \oplus
\mathcal{A}_2\oplus\dots \oplus \mathcal{A}_{k}$, where $k$ is nilindex of $\mathcal{A}$. Using $\mathcal{A}^i\mathcal{A}^j\subseteq \mathcal{A}^{i+j}$, it is easy to establish that $\mathcal{A}_i\mathcal{A}_j\subseteq \mathcal{A}_{i+j}$ and we obtain the graded algebra $\gr \mathcal{A}$. If $\gr \mathcal{A}$ and $\mathcal{A}$ are isomorphic,
denoted by $\gr \mathcal{A}\cong \mathcal{A}$, we say that the algebra $\mathcal{A}$ is naturally
graded.

For any element $x$ of $\mathcal{A}$ we define the  left multiplication operator as
\[L_x \colon  \mathcal{A} \rightarrow \mathcal{A},  \quad z \mapsto xz, \ z\in\mathcal{A}.\]

Let us $x\in\mathcal{A}\setminus\mathcal{A}^2$ and for the nilpotent left multiplication operator $L_x$, define the decreasing sequence $C(x)=(n_1,n_2, \dots, n_k)$ that consists of the dimensions of the Jordan blocks of the operator $L_x$. Endow the set of these sequences with the lexicographic order, i.e. $C(x)=(n_1,n_2, \dots, n_k)\leq C(y)=(m_1,m_2, \dots, m_s)$ means that there is an $i\in\mathbb{N}$ such that $n_j=m_j$ for all $j<i$ and $n_i<m_i$.

\begin{de} The sequence $C(\mathcal{A})=\max_{x\in\mathcal{A}\setminus\mathcal{A}^2}C(x)$ is called the characteristic sequence of the algebra $\mathcal{A}$.
\end{de}

\section{Main result} \label{S:fil}

Let $\mathcal{A}$ is $n$ - dimensional associative algebra with characteristic sequence $C(\mathcal{A})=(n_1,n_2,\dots,n_k)$. Then there exists
a basis $\{e_1,e_2,\dots, e_{n}\}$ such that $e_1 \in \mathcal{A}\backslash\mathcal{A}^2$ and the matrix of the operator of left multiplication by an element $e_1$ has the form
\[L_{e_1,\sigma}=\begin{pmatrix}
J_{n_{\sigma(1)}}\\
&J_{n_{\sigma(2)}}\\
&&\ddots\\
&&&J_{n_{\sigma(k)}}
\end{pmatrix}\]
where $\sigma(i)$ belongs to $\{1,2,\dots,k\}$.

By a suitable permutation of basis elements, we can guarantee that
\[n_{\sigma(2)}\geq n_{\sigma(3)}\geq \dots \geq n_{\sigma(k)}.\]

Let $\mathcal{A}$ be a naturally graded associative algebra with characteristic sequence equal to $(n_1,n_2,\dots,n_k)$.

\begin{pr}\label{prom}
There are no naturally graded nilpotent associative algebras with $n_{\sigma(2)}>n_{\sigma(1)}+1$.
\end{pr}

\begin{proof}
Let $\mathcal{A}$ be naturally graded nilpotent associative algebra with the operator of left multiplication given above.

Then there exists
a basis $\{e_1^{(1)},e_2^{(1)},\dots, e_{i_1}^{(1)}, e_1^{(2)},e_2^{(2)},\dots, e_{i_2}^{(2)},\dots, e_1^{(n)},e_2^{(n)},\dots, e_{i_n}^{(n)}\}$ (so-called adapted basis) such that

\[\left\{\begin{array}{ll}
e_1^{(1)}e_j^{(p)}=e_{j+1}^{(p)}, & 1\leq j\leq i_p-1, 1\leq p\leq n,\\
e_1^{(1)}e_{i_p}^{(p)}=0, & 1\leq p\leq n.
\end{array}\right. \]


By induction, we can deduce that
\[\left\{\begin{array}{ll}
e_r^{(1)}e_j^{(p)}=e_{r+j}^{(p)}, & 2\leq r+j\leq i_p, 1\leq p\leq n,\\
e_r^{(1)}e_j^{(p)}=0, &  r+j\geq i_p+1, 1\leq p\leq n,
\end{array}\right. \]


 Let us suppose $i_2>i_1+1$. Then, from the chain of equalities
\[e_{i_1+2}^{(2)}=e_1^{(1)}e_{i_1+1}^{(2)}=e_1^{(1)}(e_{i_1}^{(1)}e_1^{(2)})=(e_1^{(1)}e_{i_1}^{(1)})e_1^{(2)}=0,\]
we have contradiction. Consequently, there are no such type algebras for $i_2>i_1+1$.
\end{proof}

\begin{remark} In the case when $i_2=i_1+1$ investigated in \cite{kar3}. Actually, in \cite{kar3} the author classified naturally graded associative algebras with the operator $L_{e_1}$ has the following form
\[\begin{pmatrix}
J_{m}&0\\
0&J_{m+1}
\end{pmatrix}.\]

\end{remark}

Let $\mathcal{A}$  be a naturally graded $n$-dimensional associative algebra with characteristic sequence $C(\mathcal{A})=(n-3,2,1)$. Then there exists
a basis $\{e_1,e_2,\dots, e_{n}\}$ such that $e_1 \in \mathcal{A}\backslash\mathcal{A}^2$ and $C(e_1)=(n-3,2,1)$.

By the definition of characteristic sequence the operator $L_{e_1}$ in the Jordan form has one block $J_{n-3}$ of size $n-3$, second block $J_2$ of size $2$ and third block $J_1$ of size one.

The possible forms for the operator $L_{e_1}$ are six. By changing basis it can be reduced to the following three cases:
\[\begin{pmatrix}
J_{n-3}&0&0\\
0&J_2&0\\
0&0&J_1\\
\end{pmatrix},\quad
\begin{pmatrix}
J_{2}&0&0\\
0&J_{n-3}&0\\
0&0&J_1\\
\end{pmatrix}, \quad
\begin{pmatrix}
J_{1}&0&0\\
0&J_{n-3}&0\\
0&0&J_2\\
\end{pmatrix}. \]

By Proposition \ref{prom} it is easy to see for $n>6$ there are not exist any naturally graded associative algebras which the operator of the left multiplication has the last two cases. Therefore, it is sufficient to consider only first case for $n>6$.

%
%

Let $\mathcal{A}$ be naturally graded nilpotent associative algebra with characteristic sequence $(n-3,2,1)$ for $n>6$. Then there exists
a basis $\{e_1,e_2,\dots, e_{n}\}$ (so-called adapted basis) such that
\[\left\{\begin{array}{ll}
e_1e_i=e_{i+1}, & 1\leq i\leq n-4,\\
e_1e_{n-3}=0,\\
e_1e_{n-2}=e_{n-1},\\
e_1e_{n-1}=0,\\
e_1e_n=0.
\end{array}\right. \]

From these products, it is easy to see that
\[\mathcal{A}_i\supseteq \langle e_i\rangle \quad 1\leq i\leq n-3.\]

By applying associativity low and induction, we can obtain that
\begin{equation}\label{for1} \left\{\begin{array}{ll}
e_ie_j=e_{i+j}, & 2\leq i+j\leq n-3,\\
e_ie_{n-3}=0, & 1\leq i\leq n,\\
e_1e_{n-2}=e_{n-1},\\
e_ie_{n-2}=0, & 2\leq i\leq n-3,\\
e_ie_{n-1}=e_ie_n=0,& 1\leq i\leq n-3.
\end{array}\right.\end{equation}

Let us suppose that $e_{n-2}\in \mathcal{A}_{r_1}$ and $e_{n}\in \mathcal{A}_{r_2}$, then $e_{n-1}\in \mathcal{A}_{r_1+1}$.

It is clear that $dim(\mathcal{A}_1)>1$. Otherwise the algebra $\mathcal{A}$ is one generated algebra and thats why it is a null-filiform, but it is not an algebra of nilindex $n-3$. Thus, $r_1=1$ or $r_2=1$.

We consider all possible cases depending on the values $r_1$ and $r_2$.

\textbf{Case $\mathbf{r_1=r_2=1}$.} Let $r_1=r_2=1$ and $n>6$, then we have
\[\mathcal{A}_1= \langle e_1,e_{n-2},e_n\rangle, \quad \mathcal{A}_2= \langle e_2,e_{n-1}\rangle, \quad \mathcal{A}_i= \langle e_i\rangle, \quad 3\leq i\leq n-3.\]

Moreover, we have the products which given in (\ref{for1}) and we can assume
\[\begin{array}{llllll}
e_{n-2}e_1&=&\beta_1e_2+\beta_2e_{n-1}, & e_{n-2}e_2&=&\beta_{13}e_3,\\
e_{n-2}e_{n-2}&=&\beta_3e_2+\beta_4e_{n-1}, & e_{n-2}e_{n-1}&=&\beta_{14}e_3,\\
e_{n-2}e_{n}&=&\beta_5e_2+\beta_6e_{n-1},&\\
e_ne_1&=&\beta_7e_2+\beta_8e_{n-1}, & e_ne_2&=&\beta_{15}e_3,\\
e_ne_{n-2}&=&\beta_9e_2+\beta_{10}e_{n-1},& e_ne_{n-1}&=&\beta_{16}e_3,\\
e_ne_n&=&\beta_{11}e_2+\beta_{12}e_{n-1}. &\\
\end{array}\]

By verifying the associativity low on elements, we have the following restrictions.
\[\begin{array}{llll}
\text{ Identity }& & \text{ Constraints } &\\[1mm]
\hline \hline\\
(e_2,e_{n-2},e_1)=0&\Longrightarrow &\beta_1=0, &  \\[1mm]
(e_2,e_{n-2},e_{n-2})=0&\Longrightarrow &\beta_3=0, &  \\[1mm]
(e_2,e_{n-2},e_n)=0&\Longrightarrow &\beta_5=0, &  \\[1mm]
(e_1,e_n,e_1)=0&\Longrightarrow &\beta_7=0, &  \\[1mm]
(e_1,e_n,e_{n-2})=0&\Longrightarrow &\beta_9=0, &  \\[1mm]
(e_1,e_n,e_n)=0&\Longrightarrow &\beta_{11}=0, &  \\[1mm]
(e_1,e_n,e_2)=0&\Longrightarrow &\beta_{15}=0, &  \\[1mm]
(e_1,e_n,e_{n-1})=0&\Longrightarrow &\beta_{16}=0. &  \\[1mm]

\end{array}\]

By considering the next equalities
\[\begin{array}{l}
e_{n-1}e_1=(e_1e_{n-2})e_1=e_1(e_{n-2}e_1)=\beta_2e_1e_{n-1}=0,\\[1mm]
e_{n-1}e_{n-2}=(e_1e_{n-2})e_{n-2}=e_1(e_{n-2}e_{n-2})=\beta_4e_1e_{n-1}=0,\\[1mm]
e_{n-1}e_{n-1}=e_{n-1}(e_1e_{n-2})=(e_{n-1}e_1)e_{n-2}=0,\\[1mm]
e_{n-1}e_n=(e_1e_{n-2})e_n=e_1(e_{n-2}e_n)=\beta_6e_1e_{n-1}=0,\\[1mm]
\end{array}\]
we get \(e_{n-1}e_1=e_{n-1}e_{n-2}=e_{n-1}e_{n-1}=e_{n-1}e_n=0\).

From the chain of equalities
\[\beta_{13}e_3=e_{n-2}e_2=e_{n-2}(e_1e_1)=(e_{n-2}e_1)e_1=\beta_2e_{n-1}e_1=0,\]
we conclude that \(\beta_{13}=0\).

Then the equalities
\[0=e_{n-2}(e_{n-1}e_1)=(e_{n-2}e_{n-1})e_1=\beta_{14}e_4\]
imply \(\beta_{14}=0\).

It is not difficult to observe that
\[e_{n-2}e_i=e_{n-2}(e_1e_{i-1})=(e_{n-2}e_1)e_{i-1}=\beta_2e_{n-1}e_{i-1}=\beta_2e_{n-1}(e_1e_{i-2})=\beta_2(e_{n-1}e_1)e_{i-2}=0,\]
for \(2\leq i\leq n-3\).

Analogously, as in the previous case it is easy to check that
\[e_{n-1}e_i=e_ne_i=0,\]
for \(2\leq i\leq n-3\).

Summarizing and by denoting parameters, we obtain the following table of multiplication of the algebra in this case and we denote it $\mathcal{A}(\alpha_1,\alpha_2,\alpha_3,\alpha_4,\alpha_5,\alpha_6)$
\begin{equation}\label{for2} \left\{\begin{array}{ll}
e_ie_j=e_{i+j}, & 2\leq i+j\leq n-3,\\
e_1e_{n-2}=e_{n-1},\\
e_{n-2}e_1=\alpha_1e_{n-1}, \\
e_{n-2}e_{n-2}=\alpha_2e_{n-1}, \\
e_{n-2}e_{n}=\alpha_3e_{n-1},\\
e_ne_1=\alpha_4e_{n-1},\\
e_ne_{n-2}=\alpha_5e_{n-1},\\
e_ne_n=\alpha_6e_{n-1}. \\
\end{array}\right.\end{equation}
where the omitted products are vanish.

\begin{teo}\label{teo} An arbitrary associative algebra of the family $\mathcal{A}(\alpha_1,\alpha_2,\alpha_3,\alpha_4,\alpha_5,\alpha_6)$ is isomorphic to one of the following pairwise non-isomorphic algebras:
\[
\begin{array}{lll}
\mathcal{A}(\alpha,0,0,0,0,0),& \mathcal{A}(0,0,0,1,0,0),& \mathcal{A}(1,1,0,0,0,0),\\
\mathcal{A}(0,1,0,0,0,0), & \mathcal{A}(1,0,1,0,0,0), &\mathcal{A}(1,0,0,1,0,1), \\
\mathcal{A}(\beta,0,0,0,0,1),& \mathcal{A}(1,1,0,0,0,1), & \mathcal{A}(0,1,0,\mathbf{i},0,1),\\
\mathcal{A}(0,1,0,0,0,1),& \mathcal{A}(0,0,0,0,1,0),& \mathcal{A}(1,0,0,0,1,0),\\
\mathcal{A}(0,0,0,1,1,1),&\mathcal{A}(1,0,0,0,1,1),&\mathcal{A}(0,1,0,\gamma,1,\gamma(1-\gamma)),\\
\mathcal{A}(1,1,0,0,1,\delta).
\end{array}\]
where  $\alpha,\beta \in \mathds{C}, \delta\in \mathds{C}\setminus\{0\}, \gamma\in \mathds{C}\setminus\{0,1\}, $
\end{teo}

\begin{proof}
Due to the property of natural gradation of the algebra it is enough to consider the following change
of generators:
\[\begin{array}{l}
e_1^\prime= A_1e_1+A_2e_{n-2}+A_3e_n, \\
e_{n-2}^\prime=B_1e_1+B_2e_{n-2}+B_3e_n, \\
e_{n}^\prime=C_1e_1+C_2e_{n-2}+C_3e_n
\end{array}\]
and the other elements of the new basis are obtained as products of the above elements.

Then, we obtain the following restrictions:
\begin{align*}
B_1=C_1=0, \quad C_2=-\frac{\alpha_3A_2C_3+\alpha_6A_3C_3}{A_1+\alpha_2A_2+\alpha_5A_3}, \\
\frac{A_1C_3(A_1B_2+\alpha_2A_2B_2+\alpha_3A_2B_3+\alpha_5A_3B_2+\alpha_6A_3B_3)}{A_1+\alpha_2A_2+\alpha_5A_3}\neq0
\end{align*}

and parameters
\begin{align*}
\alpha_1^\prime=\frac{(\alpha_1 A_1+\alpha_2 A_2+\alpha_3 A_3)B_2+(\alpha_4 A_1+\alpha_5 A_2+\alpha_6 A_3)B_3}{A_1B_2+\alpha_2A_2B_2+\alpha_3A_2B_3+\alpha_5A_3B_2+\alpha_6A_3B_3},\\
\alpha_2^\prime=\frac{\alpha_2B_2^2+(\alpha_3+\alpha_5)B_2B_3+\alpha_6B_3^2}{A_1B_2+\alpha_2A_2B_2+\alpha_3A_2B_3+\alpha_5A_3B_2+\alpha_6A_3B_3},
\end{align*}
\begin{align*}
\alpha_3^\prime=\frac{((\alpha_3B_2+\alpha_6B_3)A_1+(\alpha_2\alpha_6-\alpha_3\alpha_5)(A_2B_3-A_3B_2))C_3}{(A_1+\alpha_2A_2+\alpha_5A_3)(A_1B_2+\alpha_2A_2B_2+\alpha_3A_2B_3+\alpha_5A_3B_2+\alpha_6A_3B_3)},
\end{align*}
\begin{align*}
\alpha_4^\prime=\frac{(\alpha_4A_1^2+(\alpha_2\alpha_4-\alpha_1\alpha_3+\alpha_5)A_1A_2+\alpha_2(\alpha_5-\alpha_3)A_2^2)C_3}{(A_1+\alpha_2A_2+\alpha_5A_3)(A_1B_2+\alpha_2A_2B_2+\alpha_3A_2B_3+\alpha_5A_3B_2+\alpha_6A_3B_3)}+\\
\frac{((\alpha_4\alpha_5-\alpha_1\alpha_6+\alpha_6)A_1A_3+(\alpha_5^2-\alpha_3^2)A_2A_3+\alpha_6(\alpha_5-\alpha_3)A_3^2)C_3}{(A_1+\alpha_2A_2+\alpha_5A_3)(A_1B_2+\alpha_2A_2B_2+\alpha_3A_2B_3+\alpha_5A_3B_2+\alpha_6A_3B_3)},
\end{align*}
\begin{align*}
\alpha_5^\prime=\frac{((\alpha_5B_2+\alpha_6B_3)A_1+\alpha_2(\alpha_5-\alpha_3)A_2B_2+(\alpha_5^2-\alpha_2\alpha_6)A_3B_2)C_3}{(A_1+\alpha_2A_2+\alpha_5A_3)(A_1B_2+\alpha_2A_2B_2+\alpha_3A_2B_3+\alpha_5A_3B_2+\alpha_6A_3B_3)}+\\
\frac{((\alpha_2\alpha_6-\alpha_3^2)A_2B_3+\alpha_6(\alpha_5-\alpha_3)A_3B_3)C_3}{(A_1+\alpha_2A_2+\alpha_5A_3)(A_1B_2+\alpha_2A_2B_2+\alpha_3A_2B_3+\alpha_5A_3B_2+\alpha_6A_3B_3)},
\end{align*}
\begin{align*}
\alpha_6^\prime=\frac{(\alpha_6A_1^2+(2\alpha_2\alpha_6-\alpha_3(\alpha_5+\alpha_3))A_1A_2+\alpha_2(\alpha_2\alpha_6-\alpha_3\alpha_5)A_2^2)C_3^2}{(A_1+\alpha_2A_2+\alpha_5A_3)(A_1B_2+\alpha_2A_2B_2+\alpha_3A_2B_3+\alpha_5A_3B_2+\alpha_6A_3B_3)}+\\
\frac{(\alpha_6(\alpha_5-\alpha_3)A_1A_3+(\alpha_3\alpha_5(\alpha_5-\alpha_3)+\alpha_2\alpha_6(\alpha_3+\alpha_5))A_2A_3)C_3^2}{(A_1+\alpha_2A_2+\alpha_5A_3)(A_1B_2+\alpha_2A_2B_2+\alpha_3A_2B_3+\alpha_5A_3B_2+\alpha_6A_3B_3)}+\\
\frac{\alpha_6(\alpha_2\alpha_6-\alpha_3\alpha_5)A_3^2C_3^2}{(A_1+\alpha_2A_2+\alpha_5A_3)(A_1B_2+\alpha_2A_2B_2+\alpha_3A_2B_3+\alpha_5A_3B_2+\alpha_6A_3B_3)}.
\end{align*}

Furthermore, we obtain the next invariant expressions
\[\alpha_5^\prime-\alpha_3^\prime=\frac{(\alpha_5-\alpha_3)C_3}{A_1+\alpha_2A_2+\alpha_5A_3}, \quad \alpha_2^\prime\alpha_6^\prime-\alpha_3^\prime\alpha_5^\prime=\frac{(\alpha_2\alpha_6-\alpha_3\alpha_5)C_3^2}{(A_1+\alpha_2A_2+\alpha_5A_3)^2},\]
\[(\alpha_3^\prime+\alpha_5^\prime)^2-4\alpha_2^\prime\alpha_6^\prime=\frac{((\alpha_3+\alpha_5)^2-4\alpha_2\alpha_6)C_3^2}{(A_1+\alpha_2A_2+\alpha_5A_3)^2}.\]

\textbf{Case a.} Let $\alpha_5=\alpha_3$, then $\alpha_5^\prime=\alpha_3^\prime$ and we have
\[\alpha_2^\prime\alpha_6^\prime-(\alpha_3^\prime)^2=\frac{(\alpha_2\alpha_6-\alpha_3^2)C_3^2}{(A_1+\alpha_2A_2+\alpha_5A_3)^2}.\]

\textbf{Case a.1.} Let $\alpha_3^2=\alpha_2\alpha_6$, then $(\alpha_3^\prime)^2=\alpha_2^\prime\alpha_6^\prime$ and
\[\alpha_6^\prime=\frac{\alpha_6A_1^2C_3^2}{(A_1+\alpha_2A_2+\sqrt{\alpha_2\alpha_6}A_3)(A_1B_2+\alpha_2A_2B_2+\sqrt{\alpha_2\alpha_6}(A_2B_3+A_3B_2)+\alpha_6A_3B_3)}\]

 \textbf{Case a.1.1.} Let $\alpha_6=0$, then $\alpha_6^\prime=0$ and
\[\alpha_2^\prime=\frac{\alpha_2B_2}{A_1+\alpha_2A_2}, \quad \alpha_4^\prime=\frac{\alpha_4A_1C_3}{(A_1+\alpha_2A_2)B_2}, \quad A_1B_2C_3(A_1+\alpha_2A_2)\neq0.\]

\textbf{Case a.1.1.1.} Let $\alpha_2=0$, then $\alpha_2^\prime=0$ and
\[\alpha_4^\prime=\frac{\alpha_4C_3}{B_2}, \quad \alpha_1^\prime=\frac{\alpha_1B_2+\alpha_4B_3}{B_2}, \quad A_1B_2C_3\neq0.\]

\textbf{Case a.1.1.1.1.} Let $\alpha_4=0$, then $\alpha_4^\prime=0$ and $\alpha_1^\prime=\alpha_1$. So, we obtain $\mathcal{A}(\beta,0,0,0,0,0), \beta\in\mathds{C}$.

\textbf{Case a.1.1.1.2.} Let $\alpha_4\neq0$, then by changing $C_3=B_2/\alpha_4$ and $B_3=-\alpha_1B_2/\alpha_4$, we get $\alpha_1^\prime=0, \ \alpha_4^\prime=1$. Follows, we have $\mathcal{A}(0,0,0,1,0,0)$.

\textbf{Case a.1.1.2.} Let $\alpha_2\neq0$, then by choosing $B_2=\frac{A_1+\alpha_2A_2}{\alpha_2}$, we get $\alpha_2^\prime=1$ and
\[\alpha_4^\prime=\frac{\alpha_2\alpha_4A_1C_3}{(A_1+\alpha_2A_2)^2}, \quad A_1C_3(A_1+\alpha_2A_2)\neq0.\]

\textbf{Case a.1.1.2.1.} Let $\alpha_4=0$, then $\alpha_4^\prime=0$ and we obtain
\[\alpha_1^\prime-1=\frac{(\alpha_1-1)A_1}{A_1+\alpha_2A_1}, \quad A_1C_3(A_1+\alpha_2A_2)\neq0.\]

\textbf{Case a.1.1.2.1.1.} Let $\alpha_1=1$, then we get $\alpha_1^\prime=1$ and we have $\mathcal{A}(1,1,0,0,0,0)$.

\textbf{Case a.1.1.2.1.2.} Let $\alpha_1\neq1$, then by choosing $A_2=-\alpha_1A_2/\alpha_2$ we get $\alpha_1^\prime=0$ and we have $\mathcal{A}(0,1,0,0,0,0)$.

\textbf{Case a.1.1.2.2.} Let $\alpha_4\neq0$, then by choosing \[B_3=-\frac{\alpha_1A_1^2+\alpha_2(1+\alpha_1)A_1A_2+\alpha_2^2A_2^2}{\alpha_2\alpha_4A_1}, \quad C_3=\frac{(A_1+\alpha_2A_2)^2}{\alpha_2\alpha_4A_1}\] we obtain
$\alpha_1^\prime=0, \alpha_4^\prime=1$ and we have $\mathcal{A}(0,1,0,1,0,0)$.

 \textbf{Case a.1.2.} Let $\alpha_6\neq0$, then by choosing
 \[B_2=\frac{\alpha_6A_1C_3^2}{(A_1+\alpha_2A_2+\sqrt{\alpha_2\alpha_6}A_3)^2},\quad B_3=-\frac{\sqrt{\alpha_2}B_2}{\sqrt{\alpha_6}}\]
we obtain
\[\alpha_2^\prime=\alpha_3^\prime=\alpha_5^\prime=0, \quad \alpha_6^\prime=1.\]

Moreover, we have
\[\alpha_1^\prime=\frac{\alpha_1\sqrt{\alpha_6}-\alpha_4\sqrt{\alpha_2}}{\sqrt{\alpha_6}}, \] \[\alpha_4^\prime=\frac{(A_1+\alpha_2A_2+\sqrt{\alpha_2\alpha_6}A_3)}{\alpha_6A_1C_3}\times\]
\[\times(\alpha_4A_1+\alpha_2\alpha_4A_2+((1-\alpha_1)A_2+\alpha_4A_3)\sqrt{\alpha_2\alpha_6}+\alpha_6(1-\alpha_1)A_3).\]

\textbf{Case a.1.2.1.} Let $\alpha_1=1$, then $\alpha_1^\prime=1$ and
\[\alpha_4^\prime=\frac{\alpha_4(A_1+\alpha_2A_2+\sqrt{\alpha_2\alpha_6}A_3)^2}{\alpha_6A_1C_3}.\]

\textbf{Case a.1.2.1.1.} Let $\alpha_4=0$, then $\alpha_4^\prime=0$ and we have  $\mathcal{A}(1,0,0,0,0,1)$.

\textbf{Case a.1.2.1.2.} Let $\alpha_4\neq0$, then by choosing
\[C_3=\frac{\alpha_4(A_1+\alpha_2A_2+\sqrt{\alpha_2\alpha_6}A_3)^2}{\alpha_6A_1}\]
we obtain $\alpha_4^\prime=1$ and we have  $\mathcal{A}(1,0,0,1,0,1)$.

\textbf{Case a.1.2.2.} Let $\alpha_1\neq1$, then $\alpha_1\sqrt{\alpha_6}-\alpha_4\sqrt{\alpha_2}-\sqrt{\alpha_6}\neq0$.
\begin{itemize}
  \item If $\alpha_4=0$ then by choosing $A_3=-\sqrt{\alpha_2}A_2/\sqrt{\alpha_6}$, we have $\alpha_4^\prime=0$
  \item If $\alpha_4\neq0$ then, by choosing \[A_1=\frac{(\alpha_1\sqrt{\alpha_6}-\alpha_4\sqrt{\alpha_2}-\sqrt{\alpha_6})(\sqrt{\alpha_2}A_2+\sqrt{\alpha_6}A_3)}{\alpha_4}\]
         we have $\alpha_4^\prime=0$ and $\mathcal{A}(\alpha,0,0,0,0,1)$ where $\alpha\neq1$.
\end{itemize}

Summarizing cases {\bf a.1.2.1.1.} and {\bf a.1.2.2.} we obtain $\mathcal{A}(\alpha,0,0,0,0,1)$ for any $\alpha\in\mathbb{C}$.

\textbf{Case a.2.} Let $\alpha_3^2\neq \alpha_2\alpha_6$, then by putting
\[B_2=\frac{(\alpha_3^2-\alpha_2\alpha_6)A_2-\alpha_6A_1}{\alpha_3^2-\alpha_2\alpha_6}, \quad B_3=\frac{\alpha_3A_1+(\alpha_3^2-\alpha_2\alpha_6)A_3}{\alpha_3^2-\alpha_2\alpha_6}, \] \[C_3=\frac{A_1+\alpha_2A_2+\alpha_3A_3}{\sqrt{\alpha_3^2-\alpha_2\alpha_6}},\]
we obtain
\[\alpha_2^\prime=\alpha_6^\prime=1, \quad \alpha_3^\prime=\alpha_5^\prime=0.\]

So, we have
\[\left\{\begin{array}{ll}
e_ie_j=e_{i+j}, & 2\leq i+j\leq n-3,\\
e_1e_{n-2}=e_{n-1},\\
e_{n-2}e_1=\alpha_1^\prime e_{n-1}, \\
e_{n-2}e_{n-2}=e_{n-1}, \\
e_ne_1=\alpha_4^\prime e_{n-1},\\
e_ne_n=e_{n-1}. \\
\end{array}\right.\]
We make a generic change of basis as above, we have the next restrictions
\[B_1=C_1=0, \quad  C_2=-A_3, \quad B_3=A_3, \quad B_2=A_1+A_2, \quad C_3=A_1+A_2,\]
\[A_1((A_1+A_2)^2+A_3^2)\neq0.\]
Moreover, we get
\[\alpha_1'^\prime=\frac{\alpha_1^\prime A_1^2+(1+\alpha_1^\prime)A_1A_2+A_2^2+\alpha_4^\prime A_1A_3+A_3^2}{(A_1+A_2)^2+A_3^2},\]
\[\alpha_4'^\prime=\frac{A_1(\alpha_4^\prime(A_1+A_2)+(1-\alpha_1^\prime)A_3)}{(A_1+A_2)^2+A_3^2}.\]
Furthermore, we obtain the next invariant expressions
\[(1-\alpha_1'^\prime)^2+(\alpha_4'^\prime)^2=\frac{(1-\alpha_1^\prime)^2+(\alpha_4^\prime)^2}{(A_1+A_2)^2+A_3^2}.\]

\textbf{Case a.2.1.} Let $(1-\alpha_1^\prime)^2+(\alpha_4^\prime)^2=0$, then $(1-\alpha_1'^\prime)^2+(\alpha_4'^\prime)^2=0$  and we have
\[1-\alpha_1'^\prime=\frac{(1-\alpha_1^\prime)(A_1+A_2-\mathbf{i}(1-\alpha_1^\prime)A_3)A_1}{(A_1+A_2)^2+A_3^2}.\]

\textbf{Case a.2.1.1.} Let $\alpha_1^\prime=1$, then $\alpha_1'^\prime=1$ and  $\alpha_4'^\prime=0$. So, we have $\mathcal{A}(1,1,0,0,0,1)$.

\textbf{Case a.2.1.2.} Let $\alpha_1^\prime\neq1$, then by choosing $A_2=-(\alpha_1A_1+\mathbf{i}A_3)$ we get $\alpha_1'^\prime=0$ and  $\alpha_4'^\prime=\mathbf{i}$. Follows, we have $\mathcal{A}(0,1,0,\mathbf{i},0,1)$.

\textbf{Case a.2.2.} Let $(1-\alpha_1^\prime)^2+(\alpha_4^\prime)^2\neq0$, then $(1-\alpha_1'^\prime)^2+(\alpha_4'^\prime)^2\neq0$.
\begin{itemize}
  \item If $\alpha_4^\prime=0$, then $1-\alpha_1^\prime\neq0$ and by choosing $A_2=-\alpha_1^\prime A_1, \ A_3=0$ we get $\alpha_1'^\prime=\alpha_4'^\prime=0$.
  \item If $\alpha_4^\prime\neq0$, then we choose \(A_2=-\alpha_1^\prime A_1, \ A_3=-\alpha_4^\prime A_1\) we obtain $\alpha_1'^\prime=\alpha_4'^\prime=0$.
\end{itemize}

Follows we have $\mathcal{A}(0,1,0,0,0,1)$.

\textbf{Case b.} Let $\alpha_5-\alpha_3\neq0$, then $\alpha_5^\prime-\alpha_3^\prime\neq0$ and we have
\[\alpha_2^\prime\alpha_6^\prime-\alpha_3^\prime\alpha_5^\prime=\frac{(\alpha_2\alpha_6-\alpha_3\alpha_5)C_3^2}{(A_1+\alpha_2A_2+\alpha_5A_3)^2},\]
\[(\alpha_3^\prime+\alpha_5^\prime)^2-4\alpha_2^\prime\alpha_6^\prime=\frac{((\alpha_3+\alpha_5)^2-4\alpha_2\alpha_6)C_3^2}{(A_1+\alpha_2A_2+\alpha_5A_3)^2}.\]

\textbf{Case b.1.} Let $e_n\in R(\mathcal{A})$, then $\alpha_3=\alpha_6=0$ and $\alpha_3^\prime=\alpha_6^\prime=0$, $\alpha_5\neq0, \alpha_5^\prime\neq0$. Moreover, we have
\[\alpha_1^\prime=\frac{\alpha_1A_1B_2+\alpha_2A_2B_2+\alpha_4A_1B_3+\alpha_5A_2B_3}{(A_1+\alpha_2A_2+\alpha_5A_3)B_2}, \quad \alpha_2^\prime=\frac{\alpha_2B_2+\alpha_5B_3}{A_1+\alpha_2A_2+\alpha_5A_3},\]
\[\alpha_4^\prime=\frac{(\alpha_4A_1+\alpha_5A_2)C_3}{(A_1+\alpha_2A_2+\alpha_5A_3)B_2}, \quad \alpha_5^\prime=\frac{\alpha_5C_3}{A_1+\alpha_2A_2+\alpha_5A_3}.\]
Furthermore, we obtain the next invariant expression
\[\alpha_1^\prime\alpha_5^\prime-\alpha_2^\prime\alpha_4^\prime=\frac{(\alpha_1\alpha_5-\alpha_2\alpha_4)A_1C_3}{(A_1+\alpha_2A_2+\alpha_5A_3)^2}\]

\textbf{Case b.1.1.} Let $\alpha_1\alpha_5-\alpha_2\alpha_4=0$, then we have $\alpha_1=\frac{\alpha_2\alpha_4}{\alpha_5}$. We choose
\[A_2=-\frac{\alpha_4A_1}{\alpha_5}, \quad B_3=-\frac{\alpha_2B_2}{\alpha_5}, \quad C_3=\frac{(\alpha_5-\alpha_2\alpha_4)A_1+\alpha_5^2A_3}{\alpha_5^2}\]
and we get
\[\alpha_1^\prime=0, \quad \alpha_2^\prime=0, \quad \alpha_4^\prime=0, \quad \alpha_5^\prime=1.\]
Follows, we have $\mathcal{A}(0,0,0,0,1,0)$.

\textbf{Case b.1.2.} Let $\alpha_1\alpha_5-\alpha_2\alpha_4\neq0$, then by choosing
\[A_2=-\frac{\alpha_4A_1}{\alpha_5}, \quad A_3=\frac{(\alpha_1-1)A_1}{\alpha_5}, \quad B_3=-\frac{\alpha_2B_2}{\alpha_5}, \quad C_3=\frac{(\alpha_1\alpha_5-\alpha_2\alpha_4)A_1}{\alpha_5^2}\]
we obtain
\[\alpha_1^\prime=1, \quad \alpha_2^\prime=0, \quad \alpha_4^\prime=0, \quad \alpha_5^\prime=1.\]
Follows, we have $\mathcal{A}(1,0,0,0,1,0)$.

\textbf{Case b.2.} Let $e_n\notin R(\mathcal{A})$, then $(\alpha_3,\alpha_6)\neq(0,0)$. Then without losing generality we may assume that $\alpha_3\neq0$. Otherwise $\alpha_3=0$, $\alpha_6\neq0$ and we make change of basis as $e_{n-2}^\prime=e_{n-2}+e_n$ we obtain $\alpha_3^\prime=\alpha_6\neq0$.

\textbf{Case b.2.1.} Let $\alpha_2\alpha_6-\alpha_3\alpha_5=0$, then $\alpha_5=\frac{\alpha_2\alpha_6}{\alpha_3}$ and from $\alpha_3\neq\alpha_5$ we have $\alpha_3^2-\alpha_2\alpha_6\neq0$. Moreover, we get the next invariant expression
\[\alpha_1^\prime\alpha_6^\prime-\alpha_3^\prime\alpha_4^\prime=\frac{\alpha_3^2(\alpha_1\alpha_6-\alpha_3\alpha_4)}{(\alpha_3(A_1+\alpha_2A_2)+\alpha_2\alpha_6A_3)^2(\alpha_3A_1B_2+(\alpha_3A_2+\alpha_6A_3)(\alpha_2B_2+\alpha_3B_3))}.\]

\textbf{Case b.2.1.1.} Let $\alpha_1\alpha_6-\alpha_3\alpha_4=0$, then $\alpha_4=\frac{\alpha_1\alpha_6}{\alpha_3}$ and we choose
\[A_3=\frac{\alpha_2\alpha_6(\alpha_1A_1+\alpha_2A_2)-\alpha_3^2((1+\alpha_1)A_1+\alpha_2A_2)}{\alpha_3(\alpha_3^2-\alpha_2\alpha_6)},\]
\[B_2=\frac{\alpha_3^2\alpha_6A_1}{(\alpha_3^2-\alpha_2\alpha_6)^2}, \quad B_3=-\frac{\alpha_3^2A_1}{(\alpha_3^2-\alpha_2\alpha_6)^2}\]
\[C_3=\frac{-\alpha_3^4(A_1+\alpha_2A_2)-\alpha_2^2\alpha_6^2(\alpha_1A_1+\alpha_2A_2)-\alpha_2\alpha_3^2\alpha_6((2+\alpha_1)A_1+2\alpha_2A_2)}{\alpha_3(\alpha_3^2-\alpha_2\alpha_6)^2}.\]
Follows, we have \[\alpha^\prime_1=\alpha^\prime_2=\alpha^\prime_3=0, \ \alpha^\prime_4=\alpha^\prime_5=\alpha^\prime_6=1\]
and we obtain $\mathcal{A}(0,0,0,1,1,1)$.

\textbf{Case b.2.1.2.} Let $\alpha_1\alpha_6-\alpha_3\alpha_4\neq0$, then by choosing
\[A_2=\frac{\alpha_3\alpha_4A_1}{\alpha_3^2-\alpha_2\alpha_6}, \quad A_3=\frac{(1-\alpha_1)\alpha_3A_1}{\alpha_3^2-\alpha_2\alpha_6}, \quad B_2=\frac{\alpha_3^2\alpha_6A_1}{(\alpha_3^2-\alpha_2\alpha_6)^2},\]
\[B_3=-\frac{\alpha_3^2A_1}{(\alpha_3^2-\alpha_2\alpha_6)^2}, \quad C_3=-\frac{\alpha_3(\alpha_3^2+\alpha_2\alpha_3\alpha_4-\alpha_1\alpha_2\alpha_6)A_1}{(\alpha_3^2-\alpha_2\alpha_6)^2}\]
we obtain
\[\alpha^\prime_2=\alpha^\prime_3=\alpha^\prime_4=0, \ \alpha^\prime_1=\alpha^\prime_5=\alpha^\prime_6=1\]
and we have $\mathcal{A}(1,0,0,0,1,1)$.

\textbf{Case b.2.2.} Let $\alpha_2\alpha_6-\alpha_3\alpha_5\neq0$, then we denote by
\[\begin{array}{lll}
\nabla&=&\alpha_ 3^3 \alpha_ 4 + \alpha_ 3^2 \alpha_ 4 \alpha_ 5 - \alpha_ 1\alpha_ 3^2 \alpha_ 4 \alpha_ 5 + \alpha_ 2 \alpha_ 3 \alpha_ 4^2\alpha_ 5 - \alpha_ 1 \alpha_ 3 \alpha_ 4 \alpha_ 5^2 - \alpha_ 1\alpha_ 3^2 \alpha_ 6 - \\
&&3 \alpha_ 2 \alpha_ 3 \alpha_ 4 \alpha_ 6 +\alpha_ 1 \alpha_ 2 \alpha_ 3 \alpha_ 4 \alpha_ 6 - \alpha_ 2^2 \alpha_ 4^2 \alpha_ 6 + \alpha_ 3 \alpha_ 5 \alpha_ 6 + \alpha_ 1^2 \alpha_ 3 \alpha_ 5 \alpha_ 6 +   \\
&&\alpha_2 \alpha_ 4 \alpha_ 5 \alpha_ 6 +\alpha_ 1 \alpha_ 2 \alpha_ 4 \alpha_ 5 \alpha_ 6 -\alpha_ 1 \alpha_ 5^2 \alpha_ 6 - \alpha_ 2 \alpha_ 6^2 +  2 \alpha_ 1 \alpha_ 2 \alpha_ 6^2 - \alpha_ 1^2 \alpha_ 2 \alpha_ 6^2.\end{array}\]
We have the next invariant expression
\[\nabla^\prime=\frac{\nabla A_1^2C_3^4}{(A_1+\alpha_2A_2+\alpha_5A_3)^4(A_1B_2+\alpha_2A_2B_2+\alpha_5A_3B_2+\alpha_3A_2B_3+\alpha_6A_3B_3)}\]

\textbf{Case b.2.2.1.} Let $\nabla=0$, then by choosing
\[A_3=B_3+\frac{\alpha_3A_1}{\alpha_2\alpha_6-\alpha_3\alpha_5}, \quad B_2=A_2+\frac{\alpha_6A_1}{\alpha_2\alpha_6-\alpha_3\alpha_5},\]
\[C_3=-\frac{\alpha_2\alpha_6A_1}{(\alpha_3-\alpha_5)(\alpha_2\alpha_6-\alpha_3\alpha_5)}-\frac{\alpha_2A_2+\alpha_5B_3}{\alpha_3-\alpha_5}\]
we have
\[\alpha_2^\prime=\alpha_5^\prime=1, \quad \alpha_3^\prime=0, \quad \alpha_6^\prime=\frac{\alpha_2\alpha_6-\alpha_3\alpha_5}{(\alpha_3-\alpha_5)^2}\neq0,\]
with
\[\alpha_6^\prime(\alpha_1^\prime-1)^2+(\alpha_4^\prime-\alpha_1^\prime)(\alpha_4^\prime-1)=0.\]

So we have
\begin{equation}\label{for2} \left\{\begin{array}{ll}
e_ie_j=e_{i+j}, & 2\leq i+j\leq n-3,\\
e_1e_{n-2}=e_{n-1},\\
e_{n-2}e_1=\alpha_1^\prime e_{n-1}, \\
e_{n-2}e_{n-2}=e_{n-1}, \\
e_ne_1=\alpha_4^\prime e_{n-1},\\
e_ne_{n-2}=e_{n-1},\\
e_ne_n=\alpha_6^\prime e_{n-1}, &\alpha_6^\prime\neq0\\
\end{array}\right.\end{equation}
with
\[\alpha_1^\prime-\alpha_4^\prime-\alpha_1^\prime\alpha_4^\prime+(\alpha_4^\prime)^2+\alpha_6^\prime-2\alpha_1^\prime\alpha_6^\prime+(\alpha_1^\prime)^2\alpha_6^\prime=0.\]

Now,we make the generic change of basis
\[\begin{array}{l}
e_1^\prime= A_1e_1+A_2e_{n-2}+A_3e_n, \\
e_{n-2}^\prime=B_1e_1+B_2e_{n-2}+B_3e_n, \\
e_{n}^\prime=C_1e_1+C_2e_{n-2}+C_3e_n
\end{array}\]
and we have the next expressions
\[A_3=B_3, \quad B_1=C_1=0, \quad B_2=A_1+A_2, \quad C_2=-\alpha_6^\prime B_3, \quad C_3=A_1+A_2+B_3.\]
and new parameters
\[\alpha_1'^\prime=\frac{\alpha_1^\prime A_1^2+A_2^2+(1+\alpha_1^\prime)A_1A_2+(\alpha_4^\prime A_1+\alpha_6^\prime B_3+A_2)B_3}{(A_1+A_2)^2+(A_1+A_2+\alpha_6^\prime B_3)B_3},\]
\[\alpha_4'^\prime=\frac{\alpha_4^\prime A_1^2+A_2^2+(1+\alpha_4^\prime)A_1A_2+(\alpha_4^\prime+(1-\alpha_1^\prime)\alpha_6^\prime)A_1B_3+A_2B_3+\alpha_6^\prime B_3^2}{(A_1+A_2)^2+(A_1+A_2+\alpha_6^\prime B_3)B_3},\]
\[\alpha_6'^\prime=\alpha_6^\prime\neq0.\]

It is easy to prove that:
\[\alpha_1'^\prime-\alpha_4'^\prime-\alpha_1'^\prime\alpha_4'^\prime+(\alpha_4'^\prime)^2+\alpha_6'^\prime-2\alpha_1'^\prime\alpha_6'^\prime+(\alpha_1'^\prime)^2\alpha_6'^\prime=\]
\[=\frac{(\alpha_1^\prime-\alpha_4^\prime-\alpha_1^\prime\alpha_4^\prime+(\alpha_4^\prime)^2+\alpha_6^\prime-2\alpha_1^\prime\alpha_6^\prime+(\alpha_1^\prime)^2\alpha_6^\prime)A_1^2}{(A_1+A_2)^2+(A_1+A_2+\alpha_6^\prime B_3)}=0.\]

By choosing
\[A_2=-\frac{1}{2}(A_1+\alpha_1^\prime A_1+B_3-\sqrt{(A_1+\alpha_1^\prime A_1+B_3)^2-4(\alpha_1^\prime A_1^2+\alpha_4^\prime A_1B_3+\alpha_6^\prime B_3^2)})\]
we obtain $\alpha_1'^\prime=0$ and $\alpha_6'^\prime=(1-\alpha_4'^\prime)\alpha_4'^\prime\neq0$.

Thus, we have the family $\mathcal{A}(0,1,0,\gamma,1,(1-\gamma)\gamma)$ with $\gamma\in\mathds{C}\setminus\{0,1\}$.

\textbf{Case b.2.2.2.} Let $\nabla\neq0$, then by choosing
\[A_2=\frac{\alpha_4A_1}{\alpha_3-\alpha_5}, \quad A_3=\frac{(1-\alpha_1)A_1}{\alpha_3-\alpha_5},\quad B_2=\frac{(\alpha_3(\alpha_4\alpha_5-\alpha_6)+\alpha_6(\alpha_5-\alpha_2\alpha_4))A_1}{(\alpha_3-\alpha_5)(\alpha_3\alpha_5-\alpha_2\alpha_6)},\]
\[B_3=\frac{(\alpha_3^2-\alpha_1\alpha_3\alpha_5+(\alpha_1-1)\alpha_2\alpha_6)A_1}{(\alpha_3-\alpha_5)(\alpha_3\alpha_5-\alpha_2\alpha_6)},\quad C_3=\frac{(\alpha_1\alpha_5-\alpha_2\alpha_4-\alpha_3)A_1}{(\alpha_3-\alpha_5)^2}\]
we obtain
\[Det=-\frac{\nabla^2A_1^{\frac{n^2-5n+14}{2}}}{(\alpha_3-\alpha_5)^5(\alpha_3\alpha_5-\alpha_2\alpha_6)^2},\]
\[\alpha_1^\prime=\alpha_2^\prime=\alpha_5^\prime=1, \quad \alpha_3^\prime=\alpha_4^\prime=0, \quad \alpha_6^\prime=\frac{\alpha_2\alpha_6-\alpha_3\alpha_5}{(\alpha_3-\alpha_5)^2}\neq0. \]
So, we have $\mathcal{A}(1,1,0,0,1,\delta)$ for $\delta\neq0$.
\end{proof}

\textbf{Case $\mathbf{r_1=1, r_2=2}$.} Let $r_1=1, r_2=2$ and $n>6$ then we have
\[\mathcal{A}_1= \langle e_1,e_{n-2}\rangle, \quad \mathcal{A}_2= \langle e_2,e_{n-1},e_n\rangle, \quad \mathcal{A}_i= \langle e_i\rangle, \quad 3\leq i\leq n-3.\]

Moreover, we have the products which given in (\ref{for1}) and we can assume
\[\begin{array}{ll}
e_{n-2}e_1=\alpha_1e_1+\beta_1e_{n-1}+\beta_2e_n, & e_ne_1=\alpha_6e_3, \\
e_{n-2}e_{n-2}=\alpha_2e_1+\beta_3e_{n-1}+\beta_4e_n, & e_ne_{n-2}=\alpha_7e_3, \\
e_{n-2}e_{n-1}=\alpha_3e_3, & e_ne_{n-1}=\alpha_8e_3,\\
e_{n-2}e_n=\alpha_4e_3, & e_ne_n=\alpha_9e_3.\\
e_{n-1}e_1=\alpha_5e_3
\end{array}\]

From the equalities
\[\begin{array}{l}
0=(e_1e_{n-1})e_1=e_1(e_{n-1}e_1)=\alpha_5e_1e_3=\alpha_5e_4, \\
0=e_{n-2}(e_{n-1}e_1)=(e_{n-2}e_{n-1})e_1=\alpha_3e_3e_1=\alpha_3e_4, \\
0=(e_1e_n)e_1=e_1(e_ne_1)=\alpha_5e_1e_3=\alpha_6e_4, \\
0=e_{n-2}(e_ne_1)=(e_{n-2}e_n)e_1=\alpha_4e_3e_1=\alpha_4e_4,\\
0=(e_1e_n)e_{n-2}=e_1(e_ne_{n-2})=\alpha_7e_1e_3=\alpha_7e_4, \\
0=(e_1e_n)e_{n-1}=e_1(e_ne_{n-1})=\alpha_8e_1e_3=\alpha_8e_4, \\
0=(e_1e_n)e_n=e_1(e_ne_n)=\alpha_9e_1e_3=\alpha_9e_4
\end{array}\]
we obtain that
\[\alpha_3=\alpha_4=\alpha_5=\alpha_6=\alpha_7=\alpha_8=\alpha_9=0\]
follows we have
\[e_{n-2}e_{n-1}=e_{n-2}e_n=e_{n-1}e_1=e_ne_1=e_ne_{n-2}=e_ne_{n-1}=e_ne_n=0.\]
Let's consider
\[e_{n-1}e_{n-2}=(e_1e_{n-2})e_{n-2}=e_1(e_{n-2}e_{n-2})=\alpha_2e_1e_1+\beta_3e_1e_{n-1}+\beta_4e_1e_n=\alpha_2e_2,\]
but by the graduation rule  $e_{n-1}e_{n-2}\in\langle e_3\rangle$. Thus, $e_{n-1}e_{n-2}=0$ and  $\alpha_2=0$.

Now from the next chain of equalities
\[0=e_{n-2}e_{n-1}=e_{n-2}(e_1e_{n-2})=(e_{n-2}e_1)e_{n-2}=(\alpha_1e_1+\beta_1e_{n-1}+\beta_2e_n)e_{n-2}=\alpha_1e_{n-1},\]
we deduce that $\alpha_1=0$.

By induction and associativity low it is easy to proof that
\[e_{n-1}e_i=e_ne_i=e_{n-2}e_j=0, \quad 1\leq i\leq n-3, \quad 2\leq j\leq n-3.\]

Analogously, from identities
\[\begin{array}{c}
e_{n-1}e_{n-1}=e_{n-1}(e_1e_{n-2})=(e_{n-1}e_1)e_{n-2}=0, \\
e_{n-1}e_n=(e_1e_{n-2})e_n=e_1(e_{n-2}e_n)=0 \end{array}\]
we get  $e_{n-1}e_{n-1}=e_{n-1}e_n=0$.

Follows, we have the following family:

\begin{equation}\label{for3}\mathcal{A}(\beta_1,\beta_2,\beta_3,\beta_4): \left\{\begin{array}{ll}
e_ie_j=e_{i+j}, \quad 2\leq i+j\leq n-3,\\
e_1e_{n-2}=e_{n-1},\\
e_{n-2}e_1=\beta_1e_{n-1}+\beta_2e_n, \\
e_{n-2}e_{n-2}=\beta_3e_{n-1}+\beta_4e_n, \text{with} \  \ (\beta_2,\beta_4)\neq(0,0) \\
\end{array}\right.\end{equation}
where the omitted products are vanish.

\begin{teo}\label{teo1} An arbitrary associative algebra of the family $\mathcal{A}(\beta_1,\beta_2,\beta_3,\beta_4)$ is isomorphic to one of the following pairwise non-isomorphic algebras:
\[
\begin{array}{llll}
\mathcal{A}(0,1,0,0),& \mathcal{A}(0,1,1,0),& \mathcal{A}(1,0,0,1), & \mathcal{A}(\beta,1,0,1), \\
\end{array}\]
where  $\beta \in \mathds{C}. $
\end{teo}

\begin{proof} The proof is carried out by applying similar arguments to those used in Theorem \ref{teo}.
\end{proof}

\textbf{Case $\mathbf{r_1=1, r_2\geq3}$.}\begin{pr}\label{pr1}
Let $\mathcal{A}$ be naturally graded nilpotent associative algebra with characteristic sequence $(n-3,2,1)$ for $n>6$ and $r_1=1, r_2\geq3$. Then, there are no such type naturally graded associative algebras.
\end{pr}

\begin{proof}
Let $r_1=1, r_2\geq3$ and $n>6$ then we have
\[\mathcal{A}_1= \langle e_1,e_{n-2}\rangle, \quad \mathcal{A}_2= \langle e_2,e_{n-1}\rangle, \quad \dots \quad \mathcal{A}_{r_2}= \langle e_{r_2},e_n\rangle,\quad \dots \quad \mathcal{A}_{n-3}= \langle e_{n-3}\rangle. \]


Let us suppose that
\[e_{n-2}e_1=\alpha_1e_2+\alpha_2e_{n-1}, \quad e_{n-2}e_{n-2}=\alpha_3e_2+\alpha_4e_{n-1}.\]

From the next chain of equalities
\[\begin{array}{c}
0=(e_1e_{n-1})e_1=e_1(e_{n-1}e_1)=e_1((e_1e_{n-2})e_1)=e_1(e_1((e_{n-2}e_1)))=\\
e_2(e_{n-2}e_1)=e_2(\alpha_1e_2+\alpha_2e_{n-1})=\alpha_1e_4,\\
0=(e_2e_{n-2})e_{n-2}=e_2(e_{n-2}e_{n-2})=e_2(\alpha_3e_2+\alpha_4e_{n-1})=\alpha_3e_4
\end{array}\]
we obtain $\alpha_1=\alpha_3=0$, follows
\[e_{n-2}e_1=\alpha_2e_{n-1}, \quad e_{n-2}e_{n-2}=\alpha_4e_{n-1}.\]

Analogously from
\[\begin{array}{c}
e_{n-1}e_{n-2}=(e_1e_{n-2})e_{n-2}=e_1(e_{n-2}e_{n-2})=\alpha_4e_1e_{n-1}=0, \\
e_{n-2}e_{n-1}=e_{n-2}(e_1e_{n-2})=(e_{n-2}e_1)e_{n-2}=\alpha_2e_{n-1}e_{n-2}=0, \\
e_{n-1}e_1=(e_1e_{n-2})e_1=e_1(e_{n-2}e_1)=\alpha_2e_1e_{n-1}=0,\\
e_{n-2}e_2=e_{n-2}(e_1e_1)=(e_{n-2}e_1)e_1=\alpha_2e_{n-1}e_1=0,\\
e_{n-1}e_{n-1}=(e_1e_{n-2})e_{n-1}=e_1(e_{n-2}e_{n-1})=0
\end{array}\]
we have
\[e_{n-1}e_{n-2}=e_{n-2}e_{n-1}=e_{n-1}e_1=e_{n-2}e_2=e_{n-1}e_{n-1}=0.\]

By induction
\[e_{n-2}e_i=e_{n-2}(e_2e_{i-2})=(e_{n-2}e_2)e_{i-2}=0, \quad 3\leq i\leq n-3\]
we proof that
\[e_{n-2}e_i=0, \quad 3\leq i\leq n-3.\]

Similarly, we can show that
\[e_{n-1}e_i=0, \quad 2\leq i\leq n-3.\]

Summarizing, we can conclude that $e_{n-2}e_i, e_ie_{n-2}, e_{n-1}e_i, e_ie_{n-2}\in\langle e_{n-1}\rangle$ for $1\leq i\leq n-1$. Thus, it is not possible to obtain the element $e_n$ and it contradicts the gradation. Therefore, there is no any naturally graded associative algebra in this case.
\end{proof}

\textbf{Case $\mathbf{r_1\geq2, r_2=1}$.}\begin{pr} Let $\mathcal{A}$ be naturally graded nilpotent associative algebra with characteristic sequence $(n-3,2,1)$ for $n>6$ and $r_1\geq2, r_2=1$. Then, there are no such type naturally graded associative algebras.
\end{pr}

\begin{proof} The proof follows by the similar way as the proof of Proposition \ref{pr1}.
\end{proof}

\end{document}